\documentclass[a4paper,10pt,reqno]{amsart}

\usepackage{amsmath,amssymb,amscd,amsxtra,amsthm,verbatim,textcomp}
\usepackage{tikz-cd}

\usepackage[scr=dutchcal]{mathalfa}

\newtheorem{Thm}{Theorem}[section]
\newtheorem{Prop}[Thm]{Proposition}
\newtheorem{Lem}[Thm]{Lemma}
\newtheorem{Cor}[Thm]{Corollary}

\DeclareMathOperator{\Hom}{Hom}

\DeclareMathOperator{\Ext}{Ext}

\DeclareMathOperator{\Ker}{Ker}
\DeclareMathOperator{\Coker}{Coker}

\DeclareMathOperator{\add}{add}

\newcommand{\cH}{\mathbscr H}
\newcommand{\cT}{\mathbscr T}
\newcommand{\cD}{\mathbscr D}
\newcommand{\cC}{\mathbscr C}
\newcommand{\cE}{\mathbscr E}
\newcommand{\thick}{\mathrm{thick}(\cH)}

\newcommand{\SHC}{\mathbf{SH}}
\newcommand{\Ab}{\mathbf{Ab}}

\newcommand{\AxiomA}{TR4'}
\newcommand{\AxiomB}{TR4''}

\tikzset{commutative diagrams/column sep/normal=30pt}

\author{Andrew Hubery}
\address{Andrew Hubery\\Bielefeld University}
\email{hubery@math.uni-bielefeld.de}
\subjclass[2010]{18E30,18E10}
\title[Characterising $\cD^b(\cH)$]{Characterising the bounded derived category of an hereditary abelian category}

\usepackage{hyperref}

\begin{document}

\maketitle

\section{Introduction}

The bounded derived category $\cD$ of an hereditary abelian category $\cH$ has a fairly lucid structure: as an additive category, $\cD$ is the additive closure of the union of shifts $\bigcup_n\cH[n]$, and there are no morphisms from $\cH[r]$ to $\cH[s]$ for $s\neq r,r+1$. An old preprint of Ringel \cite{Ringel} aimed to characterise such categories amongst all triangulated categories, but it was not shown the functor constructed in that paper was exact. In \cite{BvR} it is shown that Ringel's result holds true provided the triangulated category is known to be the bounded derived category of an abelian category, and in a recent preprint \cite{Chen-Ringel} this is extended to algebraic triangulated categories, so the stable category of a Frobenius category. We remark that both of these results rely on the notion of a filtered triangulated category, as introduced in \cite{BBD}. In fact, as pointed out by van Roosmalen, there is a much easier proof in the case of an algebraic triangulated category, using a theorem of Keller and Vossieck \cite{KV}. We give this argument in \hyperref[Sec:alg]{Section \ref*{Sec:alg}}.

The main result of this paper is to show that we can drop this assumption on the triangulated category being algebraic.

\begin{Thm}
Let $\cT$ be a triangulated category and $\cH\subset\cT$ an admissible hereditary abelian subcategory. Then we can lift the embedding $\cH\to\cT$ to a triangle equivalence $F\colon\cD\xrightarrow\sim\thick$ from the bounded derived category of $\cH$ to the thick subcategory of $\cT$ generated by $\cH$.
\end{Thm}

This yields the following corollary, completing the missing implication in Theorem 1 of \cite{Ringel} (c.f. \cite[Theorem 1.2]{BvR} and \cite[Theorem 3.3]{Chen-Ringel}). It also answers in the negative the question posed in an earlier version of \cite{Chen-Ringel} about the existence of non-algebraic hereditary triangulated categories.

\begin{Cor}
Let $\cT$ be a triangulated category, and $\cH$ a full additive subcategory such that $\Hom(\cH,\cH[-n])=0$ for all $n>0$. Then the following are equivalent.
\begin{enumerate}
\item $\cH$ is hereditary abelian and the inclusion $\cH\to\cT$ can be lifted to an equivalence $\cD^b(\cH)\cong\cT$.
\item $\add\big(\bigcup_n\cH[n]\big)=\cT$.
\end{enumerate}
\end{Cor}

We also have the following characterisation in terms of the existence of a split t-structure (c.f. \cite[Proposition 4.2]{BvR}).

\begin{Cor}
A triangulated category is equivalent to the bounded derived category of an hereditary abelian category if and only if it admits a split, bounded t-structure.
\end{Cor}

Observe that, using the result of Keller and Vossieck, we can actually weaken the assumption on $\cH$ to just being exact, rather than abelian. It is therefore plausible that our main theorem could also extend to this situation.

We also remark that Neeman has an alternative approach to this (and other) problems. In \cite{Neeman1} he introduces a new axiom scheme for triangulated categories, which in particular ensure that the cone of a morphism is unique up to canonical isomorphism. In this setting, if $\cH$ is the heart of a t-structure on such a triangulated category $\cT$, then there is always a functor $\cD^b(\cH)\to\cT$ extending the inclusion $\cH\to\cT$.

We include two appendices. The first gives a useful lemma concerning direct sums of exact triangles which we use several times in the proof; the second collects some equivalent formulations of the Octahedral Axiom.

I would like to thank Henning Krause for drawing my attention to this problem, and to Henning, Greg Stevenson and Paul Balmer for interesting discussions on this (and other) issues. I would also like to thank Adam van Roosmalen for outlining the proof in the algebraic triangulated category setting which we give here.

\section{Admissible subcategories}

Let $\cT$ be a triangulated category. An admissible subcategory $\cC\subset\cT$ is a full additive subcategory satisfying
\[ \Hom_\cT(\cC,\cC[-n])=0 \quad\textrm{for all }n>0, \]
and closed under extensions, so given an exact triangle
\[ \begin{tikzcd} X \arrow{r} & Y \arrow{r} & Z \arrow{r} & X[1] \end{tikzcd} \]
with $X,Z\in\cC$, then also $Y\in\cC$.

\begin{Thm}[\cite{Dyer}]
An admissible subcategory $\cC\subset\cT$ inherits the structure of an exact category,  by taking as exact sequences in $\cC$
\[ \begin{tikzcd} 0 \arrow{r} & X \arrow{r}{f} & Y \arrow{r}{g} & Z \arrow{r} & 0 \end{tikzcd} \]
those for which there exists an exact triangle $(f,g,h)$ in $\cT$
\[ \begin{tikzcd} X \arrow{r}{f} & Y \arrow{r}{g} & Z \arrow{r}{h} & X[1] \end{tikzcd} \]
In this case there is an isomorphism $\Ext^1_\cC(Z,X)\xrightarrow\sim\Hom_\cT(Z,X[1])$ sending the class of $(f,g)$ to the morphism $h$.
\end{Thm}

Note that the isomorphisms $\Ext^1_\cC(Z,X)\xrightarrow{\sim}\Hom_\cT(Z,X[1])$ are compatible with pull-backs and push-outs. Explicitly, if $\epsilon\in\Ext^1_\cC(Z,X)$ and $\theta\colon Z'\to Z$, then the pull-back $\epsilon\theta\in\Ext^1_\cC(Z',X)$ is sent to the composition $\epsilon\theta\in\Hom_\cT(Z',X[1])$, and similarly if $\phi\colon X\to X'$, then the push-out $\phi\epsilon\in\Ext^1_\cC(Z,X')$ is sent to the composition $\phi[1]\epsilon\in\Hom_\cT(Z,X'[1])$.

\begin{Prop}[\cite{BBD}]
The exact structure on an admissible subcategory $\cC$ is abelian if and only if, for every map $f\colon X\to Y$ in $\cC$, there exist exact triangles
\[ \begin{tikzcd} S \arrow{r}{x} & X \arrow{r}{f} & Y \arrow{r}{y} & S[1] \end{tikzcd} \]
and
\[ \begin{tikzcd} S \arrow{r}{c[-1]} & C[-1] \arrow{r}{\theta} & K[1] \arrow{r}{k[1]} & S[1] \end{tikzcd} \]
in $\cT$ such that $xk\colon K\to X$ is a kernel for $f$ and $cy\colon Y\to C$ is a cokernel for $f$.

In this case we say that $\cC$ is an admissible abelian subcategory.
\end{Prop}

\begin{proof}
Suppose $\cC$ is abelian. Given $f\colon X\to Y$, we can factor it as $f=f_2f_1$ via its image $I$, as well as take a kernel $\iota\colon K\to X$ and a cokernel $\pi\colon Y\to C$. Now apply (TR4) to get the exact commutative diagram
\[ \begin{tikzcd}
X \arrow{r}{f_1} \arrow[-, double equal sign distance]{d} & I \arrow{r} \arrow{d}{f_2} & K[1] \arrow{r}{-\iota} \arrow{d}{k[1]} & X[1] \arrow[-, double equal sign distance]{d}\\
X \arrow{r}{f} & Y \arrow{r}{y} \arrow{d}{\pi} & S[1] \arrow{r}{-x} \arrow{d}{c} & X[1]\\
& C \arrow[-, double equal sign distance]{r} \arrow{d} & C \arrow{d}{\theta[1]}\\
& I[1] \arrow{r} & K[2]
\end{tikzcd} \]
This yields the two exact triangles in $\cT$ of the required form.

Conversely, suppose that we have such triangles for every map in $\cC$. Then $\cC$ is already preabelian, so let $f\colon X\to Y$ be a monomorphism. Since its kernel is zero we have a triangle in $\cT$
\[ \begin{tikzcd} X \arrow{r}{f} & Y \arrow{r}{g} & C \arrow{r} & X[1] \end{tikzcd} \]
where $g$ is a cokernel for $f$. It follows that the pair $(f,g)$ is exact in $\cC$, and hence $f$ is a kernel for $g$. Dually, every epimorphism is the cokernel of its kernel, so the exact category $\cC$ is abelian with its usual exact structure.
\end{proof}

Note that being an admissible abelian subcategory is stronger than being an admissible subcategory which is abelian, since it may be the case that the exact structure coming from the embedding into the triangulated category is not the exact structure coming from all kernel-cokernel pairs.

A t-structure $(\cT^{\geq0},\cT^{\leq0})$ on a triangulated category $\cT$ is a pair of full subcategories satisfying the following axioms
\begin{enumerate}
\item[(t1)] $\Hom(\cT^{\leq0},\cT^{>0})=0$.
\item[(t2)] $\cT^{\geq1}\subset\cT^{\geq0}$ and $\cT^{\leq0}\subset\cT^{\leq1}$.
\item[(t3)] for all $X$ there exists an exact triangle
\[ \begin{tikzcd} X \arrow{r} & X^+ \arrow{r}{\delta_X} & X^- \arrow{r} & X[1] \end{tikzcd} \]
with $X^+\in\cT^{>0}$ and $X^-\in\cT^{<0}$.
\end{enumerate}
Here we have written $\cT^{\geq n}:=\cT^{\geq0}[-n]$ and $\cT^{>n}:=\cT^{\geq n+1}$, and similarly $\cT^{\leq n}:=\cT^{\leq0}[-n]$ and $\cT^{<n}:=\cT^{\leq n-1}$. It follows that $(\cT^{\geq n},\cT^{\leq n})$ is again a t-structure for any $n$. Note also that $\cT^{\geq0}$ and $\cT^{\leq0}$ are closed under extensions and direct summands.

The t-structure is said to be bounded provided that
\[ \cT = \bigcup_n\cT^{[-n,n]}, \quad\textrm{where}\quad \cT^{[a,b]} := \cT^{\geq a}\cap\cT^{\leq b} \quad\textrm{for }a\leq b, \]
and split provided that $\delta_X=0$ for all $X$, in which case $X\cong X^+\oplus X^-[-1]$.

\begin{Thm}[\cite{BBD}]
If $(\cT^{\geq0},\cT^{\leq0})$ is a t-structure on $\cT$, then its heart $\cH:=\cT^{\geq0}\cap\cT^{\leq0}$ is an admissible abelian subcategory of $\cT$.
\end{Thm}

\begin{proof}
It is clear that the heart $\cH$ is an admissible subcategory. Given $f\colon X\to Y$ with $X,Y\in\cH$, let its cone be $S[1]$. Then $S\in\cT^{[0,1]}$, so applying (t3) gives $S^+\cong C[-1]$ and $S^-\cong K[1]$ for some $C,K\in\cH$. Now if a morphism $g\colon Z\to X$ in $\cH$ satisfies $fg=0$, then $g$ factors uniquely through $S\to X$, and then uniquely through $K\to S$, so that $K\to X$ is a kernel for $f$. Similarly $Y\to C$ is a cokernel for $f$, so $\cH$ is admissible abelian by the previous proposition.
\end{proof}

\begin{Prop}
A t-structure is split if and only if $\Hom(\cT^{>0},\cT^{<0})=0$.
\end{Prop}

\begin{proof}
The condition is clearly sufficient, so suppose that the t-structure is split and take $f\colon X\to Y$ with $X\in\cT^{>0}$ and $Y\in\cT^{<0}$. Set $Z=\mathrm{cone}(f)[-1]$, and write $Z=Z^+\oplus Z^-[-1]$. Using the axiom (t1) we obtain an exact triangle
\[ \begin{tikzcd} Y[-1] \arrow{r}{\binom{g}{0}} & Z[-1]\oplus Z^+ \arrow{r}{(0,h)} & X \arrow{r}{f} & Y \end{tikzcd} \]
It then follows from \hyperref[Lem:split]{Lemma \ref*{Lem:split}} that the cone of $g$ is a summand of both $X$ and $Y$, so lies in $\cT^{>0}\cap\cT^{<0}=0$. Thus $g$ is an isomorphism, so $f=0$. 
\end{proof}

\section{Admissible hereditary subcategories}

Let $\cH\subset\cT$ be an admissible subcategory. We shall say that $\cH$ is an admissible hereditary subcategory provided
\[ \Hom_\cT(\cH,\cH[n])=0 \quad\textrm{for all }n\neq0,1. \]
In this case the functor $\Ext_\cH^1(X,-)$ is right exact, and hence $\cH$ is indeed an hereditary exact category.

For example, under the canonical inclusion, an hereditary abelian category $\cH$ is an admissible hereditary abelian subcategory of its bounded derived category $\cD^b(\cH)$.

Again note that being an admissible hereditary (abelian) subcategory is stronger than just the admissible (abelian) subcategory being hereditary. For example, the heart of the usual t-structure on the stable homotopy category $\SHC$ is equivalent to the category $\Ab$ of all abelian groups, so is an admissible abelian subcategory which is also hereditary. In this case however we do not have the vanishing requirement. In fact there exist non-zero homomorphisms $X\to Y[n]$ for arbitrary large $n$ given by Steenrod operations.

\begin{Lem}
Let $\cT$ be a triangulated category with a t-structure $(\cT^{\geq0},\cT^{\leq0})$. If the t-structure is split, then the heart $\cH:=\cT^{\geq0}\cap\cT^{\leq0}$ is an admissible hereditary abelian category.
\end{Lem}

\begin{proof}
Since the t-structure is split, we have $\Hom(\cT^{>0},\cT^{<0})=0$ by the previous proposition. Since $\cH[-1]\subset\cT^{>0}$ and $\cH[n-1]\subset\cT^{<0}$ for $n\geq2$, we deduce that $\cH$ is an admissible hereditary abelian subcategory.
\end{proof}

We can now state the main result of this paper.

\begin{Thm}\label{Thm:Main}
Let $\cT$ be a triangulated category and $\cH\subset\cT$ an admissible hereditary abelian subcategory. Then we can lift the embedding $\cH\to\cT$ to a triangle equivalence $F\colon\cD^b(\cH)\xrightarrow\sim\thick$ from the bounded derived category of $\cH$ to the thick subcategory of $\cT$ generated by $\cH$.
\end{Thm}

In the case of the stable homotopy category, one can also lift the embedding $\Ab\to\SHC$ to a triangle functor $\cD^b(\Ab)\to\SHC$, the Eilenberg-Mac Lane Spectrum, but this is neither fully faithful nor essentially surjective.

The theorem yields various characterisations of those triangulated categories which are equivalent to the bounded derived category of an hereditary abelian category. Compare \cite[Theorem 2.3]{Chen-Ringel}, but observe that in their theorem the equivalence $\cD^b(\cH)\xrightarrow\sim\cT$ is only shown to be additive and commuting with the translation functors, whereas we can prove that it is always a triangle equivalence.

\begin{Cor}\label{Cor:Main}
Let $\cT$ be a triangulated category, and $\cH$ a full additive subcategory such that $\Hom(\cH,\cH[-n])=0$ for all $n>0$. Then the following are equivalent.
\begin{enumerate}
\item $\cH$ is hereditary abelian and the inclusion $\cH\to\cT$ can be lifted to an equivalence $\cD^b(\cH)\cong\cT$.
\item $\add\big(\bigcup_n\cH[n]\big)=\cT$.
\item $\cH$ is the heart of a split, bounded t-structure.
\end{enumerate}
\end{Cor}

\begin{proof}
It is clear that (1) implies (2). For (2) implies (3) note that we have a split, bounded t-structure
\[ \cT^{\leq0} := \add\big(\bigcup_{n\geq0}\cH[n]\big) \quad\textrm{and}\quad \cT^{\geq0} := \add\big(\bigcup_{n\leq0}\cH[n]\big) \]
whose heart is $\cH$. Finally, for (3) implies (1) we just need to show $\thick=\cT$.

Take $X\in\cT^{\leq0}$. By (t3) we have $X[-1]\cong X^+\oplus X^-[-1]$ with $X^+\in\cT^{>0}$ and $X^-\in\cT^{<0}$. Thus $X^+\in\cH[-1]$, say $X^+=H^0(X)[-1]$. Repeating, for all $r\geq0$ we can write $X\cong X'\oplus\bigoplus_{0\leq n\leq r}H^n(X)[-n]$ with $H^n(X)\in\cH$ and $X'\in\cT^{\leq-r}$. Since the t-structure is bounded, this process must stop. Similarly if $X\in\cT^{\geq0}$. We deduce that every object $X\in\cT$ is of the form $X\cong\bigoplus_nH^n[-n]$ with $H^n\in\cH$ almost all zero, and hence that $\cT=\thick$.
\end{proof}

\section{Algebraic triangulated categories}\label{Sec:alg}

Before giving the proof of our main theorem, we consider the special case of an algebraic triangulated category.

Recall that a Frobenius category $\cE$ is an exact category with enough projectives and enough injectives, and such that the projectives and injectives coincide \cite{Happel}. In this case we can form the stable category $\underline\cE$ by taking the quotient by the ideal of all morphisms factoring through an injective. This is then a triangulated category, and in general a triangulated category is called algebraic if it is equiavalent to the stable category of a Frobenius category. In this setting we have the following result of Keller and Vossieck \cite{KV}.

\begin{Thm}[\cite{KV}]
Let $\cE$ be a Frobenius category, $\cC$  an additive category, and $F\colon\cC\to\underline\cE$ an additive functor. Assume that $\Hom_{\underline\cE}(F\cC,F\cC[-n])=0$ for all $n>0$.
\begin{enumerate}
\item[(a)] The functor $F$ extends to a triangle functor $\tilde F\colon\mathbscr K^b(\cC)\longrightarrow\underline\cE$ from the bounded homotopy category. Moreover, $\tilde F$ is fully faithful if and only if $F$ is fully faithful and $\Hom_{\underline\cE}(F\cC,F\cC[n])=0$ for all $n>0$.
\item[(b)] Suppose that $\cC$ is exact and that $F$ sends exact sequences to exact triangles, so given an exact sequence in $\cC$,
\[ \begin{tikzcd} 0 \arrow{r} & X \arrow{r}{f} & Y \arrow{r}{g} & Z \arrow{r} & 0 \end{tikzcd} \]
there is an exact triangle in $\underline\cE$ of the form
\[ \begin{tikzcd} FX \arrow{r}{Ff} & FY \arrow{r}{Fg} & FZ \arrow{r} & FX[1] \end{tikzcd}. \]
Then $\tilde F$ factors as $\mathbscr K^b(\cC)\overset{\mathrm{can}}{\longrightarrow}\cD^b(\cC)\overset{\hat F}{\longrightarrow}\underline\cE$. Moreover, $\hat F$ is fully faithful if and only if $F$ is fully faithful and for all $\theta\colon FX'[-n]\to FX$ with $n>0$, we have $F(f)\theta=0$ for some exact sequence in $\cC$
\[ \begin{tikzcd} 0 \arrow{r} & X \arrow{r}{f} & Y \arrow{r}{g} & Z \arrow{r} & 0 \end{tikzcd}. \]
\end{enumerate}
\end{Thm}

We remark that if condition (b) holds and the functor $\hat F$ is fully faithful, then (the essential image of) $\cC$ is an admissible subcategory of $\underline\cE$.

Using this theorem we get an easy proof of \hyperref[Thm:Main]{Theorem \ref*{Thm:Main}}, and hence also an easy proof of \hyperref[Cor:Main]{Corollary \ref*{Cor:Main}}, whenever the triangulated category is algebraic. Compare \cite[Thereom 3.3]{Chen-Ringel}.

\begin{Cor}
Let $\underline\cE$ be an algebraic triangulated category and $\cH$ an admissible hereditary exact subcategory. Then we can lift the embedding $\cH\to\underline\cE$ to a triangle equivalence $\cD^b(\cH)\xrightarrow\sim\thick$.
\end{Cor}

\begin{proof}
Since $\cH$ is admissible, the embedding $\cH\to\underline\cE$ satisfies condition (b) of the above theorem. We can therefore lift the embedding to a triangle functor $\hat F\colon\cD^b(\cH)\to\underline\cE$. The essential image is clearly equal to $\thick$, so we just need to prove that $\hat F$ is fully faithful. 

Consider a morphism $\theta\colon X'[-n]\to X$ for some $n>0$. If $n>1$, then $\theta=0$ since $\cH$ is admissible hereditary. If instead $n=1$, then $\theta[1]$ corresponds to an exact sequence in $\cH$
\[ \begin{tikzcd} X \arrow{r}{f} & Y \arrow{r}{g} & X' \arrow{r} & 0 \end{tikzcd} \]
in which case clearly $f\theta=0$. Since the embedding $\cH\to\underline\cE$ is fully faithful, we deduce from part (b) again that $\hat F$ is fully faithful.
\end{proof}

\section{Proof of Theorem \ref*{Thm:Main}}

We first give the definition of the functor $F$ in the theorem. Set $\cD:=\cD^b(\cH)$, and recall that we have the inclusion functors $\iota_\cT\colon\cH\to\cT$ and $\iota_\cD\colon\cH\to\cD$, as well as the cohomology functors $H^n\colon\cD\to\cH$. In particular, we have a natural isomorphism $H^0\iota_\cD\cong\mathrm{id}_\cH$. We define the functor $F$ on objects by sending $X\in\cD$ to $\bigoplus_n(\iota_\cT H^n(X))[-n]$.

For morphisms, we first observe that for all $X,Y\in\cH$ and all $n\in\mathbb Z$ we have
\[ \Hom_\cD(X,Y[n]) \cong \Ext^n_\cH(X,Y) \cong \Hom_\cT(X,Y[n]). \]
For, this is clear if $n=0,1$, whereas if $n\neq0,1$, then all three vanish.

Next, since $\cH$ is hereditary, for each $X\in\cD$ we can fix an isomorphism
\[ \eta_X \colon X \xrightarrow\sim \bigoplus_n(\iota_\cD H^n(X))[-n]. \]
This is not functorial, but we can choose $\eta$ in such a way that $\eta_{X[1]}=\eta_X[1]$. Now, given a morphism $f\colon X\to Y$ in $\cD$, we can write $\eta_Yf\eta_X^{-1}=(f_{nm})$, where $f_{nm}\colon (\iota_\cD H^m(X))[-m]\to(\iota_\cD H^n(Y))[-n]$. We therefore define $F(f):=(f_{nm})$, but where $f_{nm}$ is now viewed as a morphism $(\iota_\cT H^m(X))[-m]\to(\iota_\cT H^n(Y))[-n]$ in $\cT$.

It is immediate that $F$ is an additive functor, and is fully faithful. Moreover, we have natural isomorphisms $H^n(X[1])\xrightarrow{\sim} H^{n+1}(X)$, leading to natural isomorphisms $(\iota_\cT H^n(X[1]))[-n]\xrightarrow{\sim}(\iota_\cT H^{n+1}(X))[-n-1]$, and hence to an isomorphism $\theta_X\colon F(X[1])\cong(F(X))[1]$. Also, by our choice of $\eta$, given $f\colon X\to Y$ we have $\eta_{Y[1]}f[1]\eta_{X[1]}^{-1}=(\eta_Yf\eta_X^{-1})[1]$, and so $\theta$ determines a natural isomorphism $F(X[1])\xrightarrow{\sim}(F(X))[1]$. Finally we have $F\iota_\cD(X)=\iota_\cT H^0(X)\cong \iota_\cT(X)$, yielding a natural isomorphism $F\iota_\cD\xrightarrow{\sim}\iota_\cT$.

We have therefore lifted the embedding $\cH\to\cT$ to an additive functor $F\colon\cD\to\cT$ which is both fully faithful and commutes with the shift. Since the essential image of $F$ is clearly equal to $\thick$, it just remains to prove that $F$ is a triangle functor.

\subsection{Proof that \textit{F} is a triangle functor}

From now on, for notational simplicity, we will suppress the inclusions $\iota$, the isomorphisms $\eta$, and the natural isomorphisms.

Given $X\in\cD$, we can write $X\cong\bigoplus_{r\leq n\leq s}X_n$ with $X_n\in\cH[n]$ and $X_r,X_s$ non-zero. The amplitude of $X$ is then defined to be $s-r$. Our proof will be done by induction on the amplitude of one of the terms in the triangle. For this, the key observation is the following, a special case of the axiom (TR4).

\begin{Lem}[Key Lemma]
Consider an exact commutative diagram in a triangulated category
\[ \begin{tikzcd}
I \arrow{r}{a} \arrow[-, double equal sign distance]{d} & X \arrow{r}{f'} \arrow{d}{f''} & Y' \arrow{r}{y'} \arrow{d}{g'} & I[1] \arrow[-, double equal sign distance]{d}\\
I \arrow{r}{f''a} & Y'' \arrow{r}{g''} \arrow{d}{y''} & Z \arrow{r}{b} \arrow{d}{z} & I\\
& E[1] \arrow[-, double equal sign distance]{r} \arrow{d}{x[1]} & E[1] \arrow{d}{(f'x)[1]}\\
& X[1] \arrow{r}{f'[1]} & Y'[1]
\end{tikzcd} \]
such that $a[1]b=h=x[1]z$. If $\Hom(Y'',g')=0$, then we also have an exact triangle
\[ \begin{tikzcd} X \arrow{r}{\binom{-f'}{f''}} & Y'\oplus Y'' \arrow{r}{(g',g'')} & Z \arrow{r}{h} & X[1] \end{tikzcd} \]
\end{Lem}

\begin{proof}
Apply (TR4) to the exact triangles coming from both rows and the first column. We get an exact commutative diagram as above, but where the second column is
\[ \begin{tikzcd} Y' \arrow{r}{\bar g'} & Z \arrow{r}{\bar z} & E[1] \arrow{r}{(f'x)[1]} & Y'[1] \end{tikzcd} \]
Moreover, $x[1]\bar z=h$ and this fits in the exact triangle
\[ \begin{tikzcd} X \arrow{r}{\binom{-f'}{f''}} & Y'\oplus Y'' \arrow{r}{(\bar g',g'')} & Z \arrow{r}{h} & X[1] \end{tikzcd} \]

By (TR3) there exists an isomorphism of exact triangles
\[ \begin{tikzcd}
Y' \arrow{r}{\bar g'} \arrow[-, double equal sign distance]{d} & Z \arrow{r}{\bar z} \arrow[dotted]{d}{\exists\theta} & E[1] \arrow{r}{(f'x)[1]} \arrow[-, double equal sign distance]{d} & Y'[1] \arrow[-, double equal sign distance]{d}\\
Y' \arrow{r}{g'} & Z \arrow{r}{z} & E[1] \arrow{r}{(f'x)[1]} & Y'[1]
\end{tikzcd} \]
Since $z\theta=\bar z$ and $x[1]\bar z=h=x[1]z$, we have $h\theta=h$. Also, $z(\theta-1)g''=\bar zg''-zg''=y''-y''=0$, so our vanishing condition implies $\theta g''=g''$. We therefore have an isomorphism of triangles
\[ \begin{tikzcd}
X \arrow{r}{\binom{-f'}{f''}} \arrow[-, double equal sign distance]{d} & Y'\oplus Y'' \arrow{r}{(\bar g',g'')} \arrow[-, double equal sign distance]{d} & Z \arrow{r}{h} \arrow{d}{\theta} & X[1] \arrow[-, double equal sign distance]{d}\\
X \arrow{r}{\binom{-f'}{f''}} & Y'\oplus Y'' \arrow{r}{(g',g'')} & Z \arrow{r}{h} & X[1]
\end{tikzcd} \]
Since the top row is exact, so too is the bottom row.
\end{proof}

Suppose we have a triangle in $\cD$
\[ \begin{tikzcd} X \arrow{r}{f} & Y \arrow{r}{g} & Z \arrow{r}{h} & X[1] \end{tikzcd} \]
Decompose $Y=Y'\oplus Y''$ such that $Y'\in\cH[n]$ and every summand of $Y''$ is in $\cH[m]$ for some $m>n$. Write $f=\binom{-f'}{f''}\colon X\to Y'\oplus Y''$ and $g=(g',g'')\colon Y'\oplus Y''\to Z$. Applying (TR4') gives us an exact commutative diagram in $\cD$ as in the Key Lemma. By induction we may assume that $F$ sends both rows and both columns to exact triangles in $\cT$. Then, since $\Hom_\cT(Y'',Y')=0$, we can apply the Key Lemma to deduce that the exact triangle in $\cD$
\[ \begin{tikzcd} X \arrow{r}{f} & Y \arrow{r}{g} & Z \arrow{r}{h} & X[1] \end{tikzcd} \]
is sent to an exact triangle in $\cT$

After rotation and shift this reduces the problem to when one term of the triangle lies in $\cH$.

Assume therefore that $X\in\cH$. If $Y=A[1]\in\cH[1]$, then also $Z=B[1]\in\cH[1]$ (it is an extension of $X[1]$ by $A[1]$). Thus the triangle in $\cD$
\[ \begin{tikzcd} X \arrow{r}{\epsilon} & A[1] \arrow{r}{a[1]} & B[1] \arrow{r}{b[1]} & X[1] \end{tikzcd} \]
corresponds to the short exact sequence in $\cH$
\[ \begin{tikzcd} \epsilon\colon 0 \arrow{r} & A \arrow{r}{a} & B \arrow{r}{b} & X \arrow{r} & 0 \end{tikzcd} \]
and hence is sent to an exact triangle in $\cT$.

Suppose instead that $Y\in\cH$. The following lemma describes precisely the exact triangles in $\cT$ (or $\cD$) coming from such a map $f\colon X\to Y$ in $\cH$.

\begin{Lem}
Let $f\colon X\to Y$ be a morphism in $\cH$. Then we have an exact triangle in $\cT$ (or $\cD$)
\[ \begin{tikzcd} X \arrow{r}{f} & Y \arrow{r}{\binom{-\epsilon}{y}} & K[1]\oplus C \arrow{r}{(x[1],\eta)} & X[1], \qquad C,K\in\cH \end{tikzcd} \]
if and only if we have an exact commutative diagram in $\cH$ (a pull-back/push-out diagram)
\[ \begin{tikzcd}{}
& & 0 \arrow{d} & 0 \arrow{d}\\
0 \arrow{r} & K \arrow{r}{x} \arrow[-, double equal sign distance]{d} & X \arrow{r}{f_1} \arrow{d}{\iota} & I \arrow{r} \arrow{d}{f_2} & 0\\
0 \arrow{r} & K \arrow{r}{\iota'} & E \arrow{r}{\pi} \arrow{d}{\pi'} & Y \arrow{r} \arrow{d}{y} & 0\\
& & C \arrow[-, double equal sign distance]{r} \arrow{d} & C \arrow{d}\\
& & 0 & 0
\end{tikzcd} \]
with $f=f_2f_1$, and where $\epsilon$ and $\eta$ are the extension classes of the middle row and middle column, respectively.
\end{Lem}

\begin{proof}
Since $\cH$ is abelian, given $f\colon X\to Y$, we can form the top row and second column; since $\cH$ is hereditary, the pull-back along $f_2$ yields an epimorphism $\Ext^1_\cH(Y,K)\to\Ext^1_\cH(I,K)$, so we can form the middle row, and hence also the first column. Let $\epsilon'=\epsilon f_2$ (pull-back) be the extension class of the first row, and $\eta'=f_1\eta$ (push-out) be the extension class of the second column. We can then form the exact commutative diagram in $\cT$
\[ \begin{tikzcd}
I \arrow{r}{f_2} \arrow[-, double equal sign distance]{d} & Y \arrow{r}{y} \arrow{d}{\epsilon} & C \arrow{r}{\eta'} \arrow{d}{\eta} & I[1] \arrow[-, double equal sign distance]{d}\\
I \arrow{r}{\epsilon'} & K[1] \arrow{r}{x[1]} \arrow{d}{\iota'[1]} & X[1] \arrow{r}{f_1[1]} \arrow{d}{\iota[1]} & I[1]\\
& E[1] \arrow[-, double equal sign distance]{r} \arrow{d}{\pi[1]} & E[1] \arrow{d}{\pi'[1]}\\
& Y[1] \arrow{r}{y[1]} & C[1]
\end{tikzcd} \]
Since $\Hom_\cT(K[1],C)=0$, we can apply the Key Lemma to get (after rotation) the exact triangle
\[ \begin{tikzcd} X \arrow{r}{f} & Y \arrow{r}{\binom{-\epsilon}{y}} & K[1]\oplus C \arrow{r}{(x[1],\eta)} & X[1] \end{tikzcd} \]
In particular, the cone of $f$ is isomorphic to $\Ker(f)[1]\oplus\Coker(f)$.

Conversely, consider any exact triangle in $\cT$
\[ \begin{tikzcd} X \arrow{r}{f} & Y \arrow{r}{\binom{-\epsilon}{y}} & K[1]\oplus C \arrow{r}{(x[1],\eta)} & X[1] \end{tikzcd} \]
where $C,K\in\cH$. Applying $\Hom_\cT(-,Z)$ for $Z\in\cH$ shows that $y$ is a cokernel for $f$, say giving the short exact sequence in $\cH$
\[ \begin{tikzcd} \eta'\colon 0 \arrow{r} & I \arrow{r}{f_2} & Y \arrow{r}{y} & C \arrow{r} & 0 \end{tikzcd} \]
We also have the short exact sequence in $\cH$ given by $\epsilon$
\[ \begin{tikzcd} \epsilon\colon 0 \arrow{r} & K \arrow{r}{\iota'} & E \arrow{r}{\pi} & Y \arrow{r} & 0 \end{tikzcd} \]
We may therefore apply (TR4') to obtain an exact commutative diagram in $\cT$ as above, which in turn yields a pull-back/push-out diagram in $\cH$.
\end{proof}

Returning to the proof of the exactness of $F$, if we have an exact triangle in $\cD$
\[ \begin{tikzcd} X \arrow{r}{f} & Y \arrow{r}{g} & Z \arrow{r}{h} & X[1] \end{tikzcd} \]
with $X,Y\in\cH$, then by the lemma above it induces a pull-back/push-out diagram in $\cH$, and hence is sent to an exact triangle in $\cT$.

More generally, consider an exact triangle in $\cD$
\[ \begin{tikzcd} X \arrow{r}{\binom{-f'}{f''}} & Y'\oplus Y'' \arrow{r}{(g',g'')} & Z \arrow{r}{h} & X[1] \end{tikzcd} \]
with $X,Y'\in\cH$ and $Y''\in\cH[1]$. As above, the cone of $f''$ is of the form $E[1]$ for some $E\in\cH$. By (TR4') this is also the cone of $g'$, so $Z$ is the cone of a morphism $E\to Y'$. In particular, $Z=Z'\oplus Z''$ with $Z'\in\cH$ and $Z''\in\cH[1]$. It therefore follows from what we have already shown that the exact triangles involving $f''$ and $g'$ are sent to exact triangles in $\cT$, as is the exact triangle involving $f'$. On the other hand, writing $g''=\binom{0}{\bar g''}\colon Y''\to Z'\oplus Z''$, we can apply \hyperref[Lem:split]{Lemma \ref*{Lem:split}} to see that the corresponding exact triangle is the direct sum of the exact triangle for $\bar g''$ together with the trivial exact triangle
\[ \begin{tikzcd} 0 \arrow{r} & Z' \arrow{r}{\sim} & Z' \arrow{r} & 0 \end{tikzcd} \]
and both of these summands are sent to exact triangles in $\cT$. We may therefore apply the Key Lemma to deduce that the exact triangle $(f,g,h)$ in $\cD$ is sent to an exact triangle in $\cT$.

Finally, given any map $f\colon X\to Y$ in $\cD$ with $X\in\cH$, we can decompose $Y=Y'\oplus Y''$ with $Y'\in\cH\oplus\cH[1]$ and $Y''\in\bigoplus_{n\neq0,1}\cH[n]$, in which case $f=\binom{f'}{0}\colon X\to Y'\oplus Y''$, and by \hyperref[Lem:split]{Lemma \ref*{Lem:split}} again, the exact triangle arising from $f$ is the direct sum of the exact triangle coming from $f'$ and the trivial exact triangle
\[ \begin{tikzcd} 0 \arrow{r} & Y'' \arrow{r}{\sim} & Y'' \arrow{r} & 0 \end{tikzcd} \]
and we already know that both of these are sent to exact triangles in $\cT$. This proves that each exact triangle in $\cD$ having one term in $\cH$ is sent to an exact triangle in $\cT$, and thus completes the proof. We conclude that $F$ is a triangle functor.

\section{Constructing admissible hereditary abelian subcategories, after C.M.~Ringel}

Let $\cT$ be an essentially small triangulated category in which every object is isomorphic to a finite direct sum of indecomposable objects. For example, any Krull-Schmidt triangulated category satisfies this, as do the bounded homotopy categories $\mathscr K^b(\mathrm{proj}\,R)$ of finitely-generated projective modules over some ring $R$. We say that $\cT$ is a block provided $\cT$ is not triangle equivalent to a non-trivial product of triangulated categories.


A path $X\rightsquigarrow Y$ between indecomposable objects in $\cT$ is a finite sequence $X=X_0,X_1,\ldots,X_n=Y$ of indecomposable objects such that for all $i$, either there is a non-zero homomorphism $X_i\to X_{i+1}$ or else $X_{i+1}=X_i[1]$. A walk is then a series of forwards or backwards paths, and being connected by a walk defines an equivalence relation on the (isomorphism classes of) indecomposable objects. The corresponding equivalence classes are called the path-connected components of $\cT$.

\begin{Lem}
The category $\cT$ is a block if and only if it is path-connected.
\end{Lem}

\begin{proof}
Take a non-zero indecomposable object $X$ and consider the set $S'$ of indecomposables reachable by walks starting from $X$. Let $S''$ be the complement of $S'$ in the set of all indecomposable objects. Define $\cT'$ and $\cT''$ to be the additive subcategories of $\cT$ generated by $S'$ and $S''$ respectively. It is easy to see that $\cT'$ and $\cT''$ are full triangulated subcategories of $\cT$, and that we have an equivalence $\cT\cong\cT'\times\cT''$. Thus if $\cT$ is a block, then $\cT=\cT'$ is path-connected.

Conversely, if $\cT$ is not a block, say having a non-trivial decomposition $\cT\cong\cT'\times\cT''$, then there is no walk from an indecomposable in $\cT'$ to an indecomposable in $\cT''$, so $\cT$ is not path-connected.
\end{proof}

\begin{Lem}
Suppose $\cT$ is a block. Then for all indecomposables $X$ and $Y$, there exists a path $X\rightsquigarrow Y[m]$ for some $m\geq0$.
\end{Lem}

\begin{proof}
Since $\cT$ is path-connected, by the previous lemma, we know that there is a walk from $X$ to $Y$, so we have a finite sequence of indecopmosables
\[ X = X_0, X_1, \ldots, X_n=Y\]
where either there is a non-zero homomorphism $X_i\to X_{i+1}$, or a non-zero homomorphism $X_{i+1}\to X_i$, or else $X_{i+1}=X_i[\pm1]$.

If for some $i$ we have $X_{i+1}=X_i[-1]$, then we replace the sequence by
\[ X_0, X_1, \ldots, X_i, X_{i+2}[1], \ldots, X_n[1]. \]
If instead there is a non-zero (and non-invertible) homomorphism $f\colon X_{i+1}\to X_i$, then we have a triangle
\[ \begin{tikzcd} X_{i+1} \arrow{r}{f} & X_i \arrow{r} & Z \arrow{r} & X_{i+1}[1] \end{tikzcd} \]
Writing $Z=Z'\oplus Z''$ with $Z'$ indecomposable, and using that $X_i$ and $X_{i+1}$ are indecomposable, we deduce from \hyperref[Lem:split]{Lemma \ref*{Lem:split}} (and its dual) that the morphisms $X_i\to Z'$ and $Z'\to X_{i+1}[1]$ are both non-zero. We therefore replace the sequence by
\[ X_0, \ldots, X_i, Z', X_{i+1}[1], \ldots, X_n[1]. \]

We repeat this process, noting that at each step the number of backwards paths decreases. Thus the process necessarily stops, and the resulting modified sequence is then a path $X\rightsquigarrow Y[m]$ for some $m\geq0$.
\end{proof}

Given an indecomposable object $M\in\cT$, consider the additive subcategories
\begin{align*}
\cT^{\leq0}_M :&= \mathrm{add}\{\textrm{indec.~$Y$ such that there is a path $M\rightsquigarrow Y$}\}\\
\cT^{\geq0}_M :&= \mathrm{add}\{\textrm{indec.~$Z$ such that there is no path $M\not\rightsquigarrow Z[-1]$}\}.
\end{align*}

\begin{Prop}
The pair $(\cT^{\leq0}_M,\cT^{\geq0}_M)$ defines a split t-structure on $\cT$.
\end{Prop}

\begin{proof}
We check the axioms of a split t-structure. Take indecomposable objects $Y\in\cT^{\leq0}_M$ and $Z\in\cT^{\geq0}_M$.
\begin{enumerate}
\item[(t1)] $\Hom(Y,Z[-1])=0$, since otherwise we have paths $M\rightsquigarrow Y\rightsquigarrow Z[-1]$, a contradiction.
\item[(t2)] $Y[1]\in\cT^{\leq0}_M$, since we have paths $M\rightsquigarrow Y\rightsquigarrow Y[1]$. Also $Z[-1]\in\cT^{\geq0}_M$, since otherwise we have paths $M\rightsquigarrow Z[-2]\rightsquigarrow Z[-1]$.
\item[(t3)] Let $X$ be indecomposable. If $X\in\cT^{\leq0}_M$, then set $X^-:=X[1]$ and $X^+:=0$. Otherwise there is no path from $M$ to $X$, so $X[1]\in\cT^{\geq0}_M$, so set $X^+:=X$ and $X^-:=0$. In both cases we have $X^+\in\cT^{>0}_M$ and $X^-\in\cT^{<0}_M$ and an exact triangle
\[ \begin{tikzcd} X \arrow{r} & X^+ \arrow{r}{0} & X^- \arrow{r} & X[1] \end{tikzcd} \qedhere \]
\end{enumerate}
\end{proof}

\begin{Prop}
Suppose that $\cT$ is a block. Then the split t-structure defined by $M$ is bounded if and only if there is no path from $M$ to $M[-1]$.
\end{Prop}

\begin{proof}
If there is a path $M\rightsquigarrow M[-1]$, then there is a path $M\rightsquigarrow M[-m]$ for all $m\geq0$. Since $\cT$ is a block, for each indecomposable $X$ there is a path $M\rightsquigarrow X[m]$ for some $m\geq0$, and hence $\cT^{\leq0}_M=0$. Thus $\cT=\cT^{\geq0}_M$.

Suppose instead that there is no path from $M$ to $M[-1]$. Given any indecomposable $X$, the set of integers $m$ for which there is a path $M\rightsquigarrow X[m]$ is bounded below. For, since $\cT$ is a block there is a path $X\rightsquigarrow M[n]$ for some $n\geq0$. Thus if there were a path $M\rightsquigarrow X[m]$ with $m<-n$, then there would be a path $M\rightsquigarrow X[m]\rightsquigarrow M[m+n]\rightsquigarrow M[-1]$, a contradiction.

Now, taking $m$ minimal such that there is a path $M\rightsquigarrow X[m]$, we see that $X[m]$ is in the heart $\cH$ of the t-structure, so $X\in\cH[-m]$. Hence the t-structure is bounded.
\end{proof}

\begin{Thm}
Assume that $\cT$ is a block, and let $M\in\cT$ be indecomposable such that there is no path from $M$ to $M[-1]$. Then $\cT$ is triangle equivalent to the bounded derived category $\cD$ of the heart $\cH$ of the t-structure determined by $M$.
\end{Thm}

\appendix

\section{Split exact triangles}

Let $\cT$ be an additive category equipped with an auto-equivalence $X\mapsto X[1]$. A triangle $(f,g,h)$ in $\cT$ is a sequence
\[ \begin{tikzcd} X \arrow{r}{f} & Y \arrow{r}{g} & Z \arrow{r}{h} & X[1] \end{tikzcd} \]
A morphism $(x,y,z)$ of triangles is a commutative diagram of the form
\[ \begin{tikzcd}
X \arrow{r}{f} \arrow{d}{x} & Y \arrow{r}{g} \arrow{d}{y} & Z \arrow{r}{h} \arrow{d}{z} & X[1] \arrow{d}{x[1]}\\
X' \arrow{r}{f'} & Y' \arrow{r}{g'} & Z' \arrow{r}{h'} & X'[1]
\end{tikzcd} \]
A pre-triangulated structure on $\cT$ consists of a collection of triangles, called exact triangles, closed under isomorphism of triangles and satisfying the following axioms
\begin{enumerate}
\item[(TR0)] For each $X$ we have an exact triangle
\[ \begin{tikzcd} 0 \arrow{r} & X \arrow{r}{1} & X \arrow{r} & 0 \end{tikzcd} \]
\item[(TR1)] For each $f\colon X\to Y$ there exists an exact triangle
\[ \begin{tikzcd} X \arrow{r}{f} & Y \arrow{r} & Z \arrow{r} & X[1] \end{tikzcd} \]
\item[(TR2)] The triangle $(f,g,h)$ is exact if and only if the triangle $(-g,-h,-f[1])$ is exact.
\item[(TR3)] Any commutative square can be completed to a morphism of exact triangles 
\[ \begin{tikzcd}
X \arrow{r}{f} \arrow{d}{x} & Y \arrow{r}{g} \arrow{d}{y} & Z \arrow{r}{h} \arrow[dotted]{d}{\exists z} & X[1] \arrow{d}{x[1]}\\
X \arrow{r}{f'} & Y' \arrow{r}{g'} & Z' \arrow{r}{h'} & X'[1]
\end{tikzcd} \]
\end{enumerate}

We record the following easy consequences of the axioms (see for example \cite{Neeman2}).

\begin{Lem}
Let $\cT$ be a pre-triangulated category.
\begin{enumerate}
\item If $(x,y,z)$ is a morphism of triangles such that two of $x,y,z$ are isomorphisms, then so is the third.
\item Each functor $\Hom(U,-)$ is homological on exact triangles, so yields a long exact sequence of abelian groups. Dually each functor $\Hom(-,U)$ is cohomological. In particular, if $(f,g,h)$ is an exact triangle, then $gf=0$.
\item The collection of exact triangles is closed under taking direct sums and direct summands.
\end{enumerate}
\end{Lem}

An exact triangle $(f,g,h)$ is split provided one of $f,g,h$ is zero.

\begin{Lem}
Given a split exact triangle
\[ \begin{tikzcd} X \arrow{r}{f} & Y \arrow{r}{g} & Z \arrow{r}{0} & X[1] \end{tikzcd} \]
there exists an isomorphism $\theta\colon Y\to X\oplus Z$ such that $\theta f=\binom10$ and $g=(0,1)\theta$.
\end{Lem}

\begin{proof}
By (TR3) there exists a morphism of exact triangles
\[ \begin{tikzcd}
X \arrow{r}{f} \arrow[-, double equal sign distance]{d} & Y \arrow{r}{g} \arrow[dotted]{d}{\exists\theta} & Z \arrow{r}{0} \arrow[-, double equal sign distance]{d} & X[1] \arrow[-, double equal sign distance]{d}\\
X \arrow{r}{\binom10} & X\oplus Z \arrow{r}{(0,1)} & Z \arrow{r}{0} & X[1]
\end{tikzcd} \]
and then $\theta$ is necessarily an isomorphism.
\end{proof}

\begin{Lem}\label{Lem:split}
Consider an exact triangle
\[ \begin{tikzcd} X \arrow{r}{\binom{f}{f'}} & Y\oplus Y' \arrow{r}{(g,g')} & Z \arrow{r}{h} & X[1] \end{tikzcd} \]
If $f'=0$, then this is isomorphic to the direct sum of the exact triangles
\[ \begin{tikzcd} X \arrow{r}{f} & Y \arrow{r}{\alpha} & C \arrow{r}{\beta} & X[1] \end{tikzcd} \]
with $C=\mathrm{cone}(f)$, and
\[ \begin{tikzcd} 0 \arrow{r} & Y' \arrow{r}{1} & Y' \arrow{r} & 0 \end{tikzcd} \]
If $g=0$ as well, then $\alpha=0$ and the first of these triangles is also split.
\end{Lem}

\begin{proof}
By (TR3) there exists a morphism of exact triangles
\[ \begin{tikzcd}[ampersand replacement=\&]
X \arrow{r}{\binom f0} \arrow[-, double equal sign distance]{d} \& Y\oplus Y' \arrow{r}{(g,g')} \arrow[-, double equal sign distance]{d} \& Z \arrow{r}{h} \arrow[dotted]{d}{\exists\binom\phi s} \& X[1] \arrow[-, double equal sign distance]{d}\\
X \arrow{r}{\binom f0} \& Y\oplus Y' \arrow{r}{\left(\begin{smallmatrix}\alpha&0\\0&1\end{smallmatrix}\right)} \& C\oplus Y' \arrow{r}{(\beta,0)} \& X[1]
\end{tikzcd} \]
which is necessarily an isomorphism. If now $g=0$, then $\alpha=\phi g=0$.
\end{proof}

\section{The Octahedral Axiom}

In this section, let $\cT$ be a pre-triangulated category. We recall the Octahedral Axiom (TR4), and two equivalent axioms (\AxiomA) and (\AxiomB). Note that (\AxiomB) can be viewed as a more precise version of (TR3) in the case where the outside maps are equalities. In fact, (TR4) is also equivalent to the \textit{a priori} weaker axiom where we do not assume the triangle given by $(\dagger)$ is exact.

Note that, using the reference \cite{Neeman2}, our axiom (TR4) is Proposition 1.4.6, our axiom (\AxiomA) is two applications of Lemma 1.4.4, and our axiom (\AxiomB) is Lemma 1.4.3.

\subsection*{(TR4)}

An exact commutative diagram
\[ \begin{tikzcd}
X \arrow{r}{f} \arrow[-, double equal sign distance]{d} & Y \arrow{r}{g} \arrow{d}{u} & Z \arrow{r}{h} \arrow[dotted]{d}{\exists u'} & X[1] \arrow[-, double equal sign distance]{d}\\
X \arrow{r}{f'} & Y' \arrow{r}{g'} \arrow{d}{v} & Z' \arrow{r}{h'} \arrow[dotted]{d}{\exists v'} & X[1]\\
& W \arrow[-, double equal sign distance]{r} \arrow{d}{w} & W \arrow{d}{w'}\\
& Y[1] \arrow{r}{g[1]} & Z[1]
\end{tikzcd} \]
can be completed as shown such that $wv'=\delta=f[1]h'$.

(Strong form) Moreover, we may even complete in such a way that we have an exact triangle
\[ \begin{tikzcd}
Y \arrow{r}{\binom{-g}{u}} & Z\oplus Y' \arrow{r}{(u',g')} & Z' \arrow{r}{\delta} & Y[1]
\end{tikzcd} \tag{$\dagger$}\]

\subsection*{(\AxiomA)}

Given an exact triangle
\[ \begin{tikzcd}
Y \arrow{r}{\binom{-g}{u}} & Z\oplus Y' \arrow{r}{(u',g')} & Z' \arrow{r}{\delta} & Y[1]
\end{tikzcd} \]
an exact commutative diagram
\[ \begin{tikzcd}
X \arrow{r}{f} \arrow[-, double equal sign distance]{d} & Y \arrow{r}{g} \arrow{d}{u} & Z \arrow{r}{h} \arrow{d}{u'} & X[1] \arrow[-, double equal sign distance]{d}\\
X \arrow{r}{f'} & Y' \arrow{r}{g'} \arrow{d}{v} & Z' \arrow[dotted]{r}{\exists h'} \arrow[dotted]{d}{\exists v'} & X[1]\\
& W \arrow[-, double equal sign distance]{r} \arrow{d}{w} & W \arrow{d}{w'}\\
& Y[1] \arrow{r}{g[1]} & Z[1]
\end{tikzcd} \]
may be completed as shown such that $wv'=\delta=f[1]h'$.

\subsection*{(\AxiomB)}

An exact commutative diagram
\[ \begin{tikzcd}
X \arrow{r}{f} \arrow[-, double equal sign distance]{d} & Y \arrow{r}{g} \arrow{d}{u} & Z \arrow{r}{h} \arrow[dotted]{d}{\exists u'} & X[1] \arrow[-, double equal sign distance]{d} \\
X \arrow{r}{f'} & Y' \arrow{r}{g'} & Z' \arrow{r}{h'} & X[1]
\end{tikzcd} \] 
may be completed as shown such that we have an exact triangle
\[ \begin{tikzcd}
Y \arrow{r}{\binom{-g}{u}} & Z\oplus Y' \arrow{r}{(u',g')} & Z' \arrow{r}{f[1]h'} & Y[1]
\end{tikzcd} \]

\subsection*{Proof}

(\AxiomB)$\Rightarrow$(TR4) (strong form). Apply (\AxiomB) first to the diagram
\[ \begin{tikzcd}
X \arrow{r}{f} \arrow[-, double equal sign distance]{d} & Y \arrow{r}{g} \arrow{d}{u} & Z \arrow{r}{h} \arrow[dotted]{d}{\exists u'} & X[1] \arrow[-, double equal sign distance]{d} \\
X \arrow{r}{f'} & Y' \arrow{r}{g'} & Z' \arrow{r}{h'} & X[1]
\end{tikzcd} \] 
and then to the diagram
\[ \begin{tikzcd}
Y \arrow{r}{\binom{-g}{u}} \arrow[-, double equal sign distance]{d} & Z\oplus Y' \arrow{r}{(u',g')} \arrow{d}{(0,1)} & Z' \arrow{r}{f[1]h'} \arrow[dotted]{d}{\exists v'} & Y[1] \arrow[-, double equal sign distance]{d} \\
Y \arrow{r}{u} & Y' \arrow{r}{v} & W \arrow{r}{w} & Y[1]
\end{tikzcd} \]
Setting $w'=g[1]w$, the resulting exact triangle
\[ \begin{tikzcd}[ampersand replacement=\&, column sep=50pt]
Z\oplus Y' \arrow{r}{\big(\begin{smallmatrix}-u'&-g'\\0&1\end{smallmatrix}\big)} \& Z'\oplus Y' \arrow{r}{(v',v)} \& W \arrow{r}{\binom{-w'}{0}} \& Z[1]\oplus Y'[1]
\end{tikzcd} \]
contains as a direct summand
\[ \begin{tikzcd}
Z \arrow{r}{-u'} & Z' \arrow{r}{v'} & W \arrow{r}{-w'} & Z[1]
\end{tikzcd} \]
so this triangle is also exact.

(TR4) (weak form)$\Rightarrow$(\AxiomA). Consider the two octahedra
\[ \begin{tikzcd}[column sep=40pt, row sep=25pt]
Y \arrow{r}{\binom{-g}{u}} \arrow[-, double equal sign distance]{d} & Z\oplus Y' \arrow{r}{(u',g')} \arrow{d}{(1,0)} & Z' \arrow{r}{\delta} \arrow[dotted]{d}{h'} & Y[1] \arrow[-, double equal sign distance]{d}\\
Y \arrow{r}{-g} & Z \arrow{r}{h} \arrow{d}{0} & X[1] \arrow{r}{f[1]} \arrow[dotted]{d}{f'[1]} & Y[1]\\
& Y'[1] \arrow[-, double equal sign distance]{r} \arrow{d}{\binom{0}{1}} & Y'[1] \arrow{d}{g'[1]}\\
& Z[1]\oplus Y'[1] \arrow{r}{(u'[1],g'[1])} & Z'[1]
\end{tikzcd} \]
and
\[ \begin{tikzcd}[column sep=40pt, row sep=25pt]
Y \arrow{r}{\binom{-g}{u}} \arrow[-, double equal sign distance]{d} & Z\oplus Y' \arrow{r}{(u',g')} \arrow{d}{(0,1)} & Z' \arrow{r}{\delta} \arrow[dotted]{d}{v'} & Y[1] \arrow[-, double equal sign distance]{d}\\
Y \arrow{r}{u} & Y' \arrow{r}{v} \arrow{d}{0} & W \arrow{r}{w} \arrow[dotted]{d}{-w'} & Y[1]\\
& Z[1] \arrow[-, double equal sign distance]{r} \arrow{d}{\binom{1}{0}} & Z[1] \arrow{d}{u'[1]}\\
& Z[1]\oplus Y'[1] \arrow{r}{(u'[1],g'[1])} & Z'[1]
\end{tikzcd} \]

(\AxiomA)$\Rightarrow$(\AxiomB). Use (TR1) to obtain an exact triangle
\[ \begin{tikzcd} Y \arrow{r}{\binom{-g}{u}} & Z\oplus Y' \arrow{r}{(a,b)} & C \arrow{r}{\delta} & Y[1] \end{tikzcd} \]
Apply (\AxiomA) to obtain (from the top two rows) an exact commutative diagram
\[ \begin{tikzcd}
X \arrow{r}{f} \arrow[-, double equal sign distance]{d} & Y \arrow{r}{g} \arrow{d}{u} & Z \arrow{r}{h} \arrow{d}{a} & X[1] \arrow[-, double equal sign distance]{d}\\
X \arrow{r}{f'} & Y' \arrow{r}{b} & C \arrow[dotted]{r}{\exists c} & X[1]
\end{tikzcd} \]
such that $f[1]c=\delta$. Use (TR3) to get an isomorphism
\[ \begin{tikzcd}
X \arrow{r}{f'} \arrow[-, double equal sign distance]{d} & Y' \arrow{r}{b} \arrow[-, double equal sign distance]{d} & C \arrow{r}{c} \arrow[dotted]{d}{\exists\gamma} & X[1] \arrow[-, double equal sign distance]{d}\\
X \arrow{r}{f'} & Y' \arrow{r}{g'} & Z' \arrow{r}{h'} & X[1]
\end{tikzcd} \]
Now set $u':=\gamma a$.

\end{document}